\documentclass[4pt]{article} 

\usepackage[utf8]{inputenc} 
\usepackage{geometry} 
\usepackage{graphicx} 

\usepackage[colorlinks,citecolor=red]{hyperref}
\usepackage{amsmath}
\usepackage{amsfonts}
\usepackage{dsfont}
\usepackage{mathrsfs}
\usepackage{latexsym}
\usepackage{graphicx}
\usepackage{amssymb}
\usepackage{float}
\usepackage{epsfig}
\usepackage{epstopdf}
\usepackage{color,xcolor}
\usepackage{booktabs} 
\usepackage{array} 
\usepackage{paralist} 
\usepackage{verbatim} 
\usepackage{subfig} 

\newtheorem{theorem}{Theorem}[section]
\newtheorem{lemma}[theorem]{Lemma}

\newtheorem{example}[theorem]{Example}

\usepackage{fancyhdr} 
\usepackage{sectsty}
\allsectionsfont{\sffamily\mdseries\upshape} 

\usepackage[nottoc,notlof,notlot]{tocbibind} 
\usepackage[titles,subfigure]{tocloft} 

\title{Signless Laplacian spectral radius and fractional matchings of graphs\thanks{Supported by National Natural Science Foundation of China (No.~11571323), Outstanding Young Talent Research Fund of Zhengzhou University (No.~1521315002), the China Postdoctoral Science Foundation (No.~2017M612410) and Foundation for University Key Teacher of Henan Province (No.~2016GGJS-007).}}

\author{Ruifang Liu$^{a}$\thanks{Corresponding author. E-mail address:~rfliu@zzu.edu.cn(R. Liu).}~~ Yu Lu$^{a}$\\~\\
{\footnotesize $^a$ School of Mathematics and Statistics, Zhengzhou
University, Zhengzhou, Henan 450001, China}}
\date{}

\begin{document}
\maketitle

\begin{abstract}
A {\it fractional matching} of a graph $G$ is a function $f$ giving each
edge a number in $[0,1]$ so that $\sum_{e\in \Gamma(v)}f(e)\leq 1$
for each $v\in V(G)$, where $\Gamma(v)$ is the set of edges
incident to $v$. The {\it fractional matching number} of $G$, written
$\alpha'_{*}(G)$, is the maximum of $\sum_{e\in E(G)}f(e)$ over all
fractional matchings $f$. In this paper, we propose the relations
between the fractional matching number and the signless Laplacian spectral radius of a graph. As applications, we also give sufficient spectral
conditions for existence of a fractional perfect matching in a graph in terms of the signless Laplacian spectral radius of the graph and its complement.


\end{abstract}

\section{Introduction}
~~~Throughout this paper, all graphs are simple and undirected. Let $G$ be a
graph with vertex set $V(G)=\{v_{1},v_{2},\ldots,v_{n}\}$ and edge set $E(G)$, and let $|V(G)|$ be the order of $G$ and $|E(G)|$ be the size of $G$.
We denote by $N(v)$ the neighbor set of a vertex $v$, and denote by $|N(v)|$ the degree of
$v.$ Let $\delta$ and $\overline{d}$ be minimum degree and average degree of $G$, respectively.
The {\em complete product} $G_{1}\bigtriangledown G_{2}$ of graphs $G_{1}$ and
$G_{2}$ is the graph obtained from $G_{1}\cup G_{2}$ by joining
every vertex of $G_{1}$ to every vertex of $G_{2}.$ For a vertex subset $S\subseteq V(G),$ the induced subgraph $G[S]$ is the
subgraph of $G$ whose vertex set is $S$ and whose edge set consists
of all edges of $G$ which have both ends in $S$. Let $G^{c}$ be the complement of $G.$

The {\em adjacency matrix} of $G$ is defined to be the matrix
$A(G)=(a_{ij})$, where $a_{ij}=1$ if $v_{i}$ is adjacent to $v_{j}$,
and $a_{ij}=0$ otherwise. Its eigenvalues can be arranged in non-increasing order as follows:
$$\lambda_{1}(G)\geq \lambda_{2}(G)\geq \cdots\geq \lambda_{n}(G)\geq 0.$$
The degree matrix of $G$ is denoted by
$D(G)=\mbox{diag}(d_{G}(v_1), d_{G}(v_2), \ldots, d_{G}(v_{n}))$.
The matrix $Q(G)=D(G)+A(G)$ is called the {\em signless Laplacian}
or the {\em $Q$-matrix} of $G.$ Note that $Q(G)$ is nonnegative,
symmetric and positive semidefinite, so its eigenvalues are real and
can be arranged in non-increasing order as follows:
$$q_{1}(G)\geq q_{2}(G)\geq \cdots\geq q_{n}(G)\geq 0,$$
where $q_{1}(G)$ is {\em the signless Laplacian spectral radius} of graph
$G.$

A \emph{matching} in a graph is a set of edges no two of which are adjacent. The \emph{matching number}
$\alpha'(G)$ is the size of a largest matching. A \emph{fractional matching} of a graph $G$ is a function
$f$ giving each edge a number in $[0,1]$ so that $\sum_{e\in
\Gamma(v)}f(e)\leq 1$ for each $v\in V(G)$, where $\Gamma(v)$ is
the set of edges incident to $v$. If $f(e)\in\{0,1\},$ for every edge $e,$ then $f$ is just a matching.
The \emph{fractional matching
number} of $G$, written $\alpha'_{*}(G)$, is the maximum of
$\sum_{e\in E(G)}f(e)$ over all fractional matchings $f$.
It was shown in \cite{SE} that $\alpha'_{*}(G)\leq n/2$.

A\emph{ fractional perfect
matching} is a fractional matching $f$ with $\sum_{e\in E(G)}
f(e)=n/2$, that is, $\alpha'_{*}(G)=n/2$. If a fractional perfect
matching takes only the values 0 and 1, it is a perfect
matching.

Let $i(G)$ stand for the numbers of isolated vertices vertices of a graph $G$.
In \cite{SE}, the following results on fractional matchings and fractional perfect
matchings were proved.
\begin{description}
\item $\bullet$ For any graph $G$, $\alpha'_{*}(G)\geq\alpha'(G).$
\item $\bullet$ If $G$ is bipartite, then $\alpha'_{*}(G)=\alpha'(G).$
\item $\bullet$ For any graph $G$, $2\alpha'_{*}(G)$ is an integer.
\item $\bullet$ (The fractional analogue of Tutte's 1-Factor Theorem) A graph $G$ has a fractional
\item ~~~perfect matching if and only if
$$i(G-S)\leq|S|$$
\item ~~~for every vertex subset $S\subseteq V(G).$
\item $\bullet$ (The fractional analogue of Berge-Tutte Formula) For any graph $G$, \begin{equation}
\alpha'_{*}(G)=\frac{1}{2}(n-\max\{i(G-S)-|S|\}),\label{e1}
\end{equation}
\item ~~~where the maximum is taken over all $S\subseteq V(G)$.
\end{description}

In recent years, the problem of finding tight sufficient spectral conditions for a graph
possessing certain properties has received much attention. Especially, the study on the relations between the eigenvalues and the matching number was initiated
by Brouwer and Haemers \cite{BA}, for regular graphs, they obtained a sufficient condition on $\lambda_{3}$ for existence of a perfect matching.
Subsequently, Cioab\v{a} et al. \cite{C, CG, CGH} refined and generalized the above result
to obtain a best upper bound on $\lambda_{3}$ to guarantee the existence of a perfect matching.
Further, O and Cioab\v{a} \cite{SO1} determined the relations between the eigenvalues of
a $t$-edge-connected $k$-regular graph and its matching number when $t\leq k-2$. Very recently, Feng and his coauthors \cite{FLH1, FLH} presented sufficient
spectral conditions of a connected graph to be $\beta$-deficient, where the deficiency of a graph is the number of vertices unmatched
under a maximum matching in $G$.

On the fractional matching number, O \cite{SO} studied the connections between the fractional matching number and the spectral radius of a connected graph
with given minimum degree. Xue et al. \cite{Xue} considered the relations between the fractional matching number and the Laplacian spectral radius of a graph.
Along this line, in this paper, we use the technique of proof in \cite{SO} and propose the relations
between the fractional matching number and the signless Laplacian spectral radius of a graph. As applications, we also give sufficient spectral
conditions for existence of a fractional perfect matching in a graph in terms of the signless Laplacian spectral radius of the graph and its complement.

\section{Relations between $\alpha'_{*}(G)$ and $q_{1}(G)$ of graphs involving
minimum degree}

~~~Before proving the main results, we list two useful lemmas. The first one is a famous result due to Dirac.

\begin{lemma}{\bf(\cite{Bd})}\label{dirac}
Let $G$ be a simple graph of minimum degree $\delta$.

\noindent{\rm{(1)}} If $\delta\geq n/2$ and $n\geq 3$, then $G$ is hamiltonian.

\noindent{\rm{(2)}} If $G$ is 2-connected and $\delta\leq n/2$, then $G$ contains a cycle of length at least $2\delta$.
\end{lemma}

Given two non-increasing real sequences
$\theta_{1}\geq \theta_{2}\geq \cdots \geq \theta_{n}$ and
$\eta_{1}\geq \eta_{2}\geq \cdots \geq \eta_{m}$
with $n>m,$ the second sequence is said to {\em interlace}
the first one if $\theta_{i}\geq \eta_{i}\geq\theta_{n-m+i}$
for $i=1, 2, \ldots, m.$
The interlacing is {\em tight} if exists an integer $k\in[0, m]$
such that $\theta_{i}=\eta_{i}$ for $1\leq i\leq k$ and $\theta_{n-m+i}=\eta_{i}$ for
$k+1\leq i\leq m.$

Consider an $n\times n$ matrix
\[
M=\left(\begin{array}{ccccccc}
M_{1,1}&M_{1,2}&\cdots &M_{1,m}\\
M_{2,1}&M_{2,2}&\cdots &M_{2,m}\\
\vdots& \vdots& \ddots& \vdots\\
M_{m,1}&M_{m,2}&\cdots &M_{m,m}\\
\end{array}\right),
\]
whose rows and columns are partitioned according to a partitioning $X_{1}, X_{2},\ldots ,X_{m}$ of $\{1,2,\ldots, n\}$. The \emph{quotient matrix} $R$ of the matrix $M$ is the $m\times m$ matrix whose entries are the
average row sums of the blocks $M_{i,j}$ of $M$. The partition is \emph{equitable} if each block $M_{i,j}$ of $M$ has constant row (and column) sum.

\begin{lemma}{\bf(\cite{BH, HW})}\label{q-in}
Let $M$ be a real symmetric matrix. Then the eigenvalues of every
quotient matrix of $M$ interlace the ones of $M.$ Furthermore, if the
interlacing is tight, then the partition is equitable.
\end{lemma}

From the Perron-Frobenius Theorem of non-negative matrices, we have the following
lemma.

\begin{lemma}{\bf(\cite{FYZ})}\label{le2}
If $H$ is a subgraph of a connected graph $G,$ then $q_{1}(G)\geq q_{1}(H).$
\end{lemma}

\begin{lemma}{\bf(\cite{CD})}\label{le3}
Let $G$ be a graph on $n$ vertices with minimum, average and maximum vertex degrees $\delta, \overline{d}$
and $\triangle.$ Then
$$2\delta \leq 2\overline{d}\leq q_{1}(G)\leq 2\triangle.$$ The equalities hold if and only if $G$ is regular.
\end{lemma}

Let $k$ be a positive integer. Let $\mathcal{B}(\delta, k)$ be the set of connected bipartite graphs with minimum degree $\delta$ and the bipartitions
$X$ and $Y$ such that:\\
(i) every vertex in $X$ has degree $\delta,$\\
(ii) $|X|=|Y|+k$, and \\
(iii) the degrees of vertices in $Y$ are equal.\\
Clearly, $|Y|\geq\delta,$ and the complete bipartite graph $K_{\delta+k, \delta}\in \mathcal{B}(\delta, k).$

The following lemma can be found in \cite{SO}. To make our paper self-contained, here we provide a distinct proof
again.

\begin{lemma}{\bf(\cite{SO})}\label{le8}
If $H\in\mathcal{B}(\delta, k),$ then $\alpha'_{*}(H)=\frac{|V(H)|-k}{2}.$
\end{lemma}

\begin{proof}
Note that $H$ is a bipartite graph, then $\alpha'_{*}(H)=\alpha'(H)=|Y|=\frac{|V(H)|-k}{2}.$\hspace*{\fill}$\Box$
\end{proof}


Next we are going to determine the signless Laplacian spectral radius of $H$ in $\mathcal{B}(\delta, k)$. The above notion of equitable partition of a vertex
set in a graph is used. Consider a partition $V(G) = V_{1}\cup V_{2}\cup\cdots \cup V_{t}$ of the vertex set of a graph $G$ into
$t$ non-empty subsets. For $1\leq i, j\leq t$, let $b_{i,j}$ denote the average number of neighbours in $V_{j}$ of the
vertices in $V_{i}.$ The {\it quotient matrix} of this partition is the $t\times t$ matrix whose $(i, j)$-th entry equals $b_{i,j}.$
This partition is {\it equitable} if for
each $1\leq i, j\leq t$, any vertex $v\in V_{i}$ has exactly $b_{i,j}$ neighbours in $V_{j}.$ In this case, the eigenvalues of the
quotient matrix are eigenvalues of $G$ and the spectral radius of the quotient matrix equals the spectral
radius of $G$ (see \cite{BH, GC} for more details).

\begin{lemma}\label{le9}
If $H\in\mathcal{B}(\delta, k),$ then $q_{1}(H)=\frac{2|V(H)|\delta}{|V(H)|-k}.$
\end{lemma}

\begin{proof}
Let $Q(H)$ be the signless Laplacian matrix of the graph $H.$ The quotient matrix $R$ of $Q(H)$ on the partitions $X$ and $Y$ is
\begin{equation*}\begin{split}
R&=\left(\begin{array}{ccccccc}
\delta&\delta\\[1mm]
\frac{\delta|X|}{|Y|}&\frac{\delta|X|}{|Y|}\\
\end{array}\right).
\end{split}\end{equation*}
By a simple calculation, the largest eigenvalue of $R$ is $\lambda_{1}(R)=\delta(1+\frac{|X|}{|Y|})=\frac{2|V(H)|\delta}{|V(H)|-k}.$
Note that the partition is equitable. By the above fact mentioned in the paragraph before Lemma \ref{le9}, we have $$q_{1}(H)=\lambda_{1}(R)=\frac{2|V(H)|\delta}{|V(H)|-k}.$$
\hspace*{\fill}$\Box$
\end{proof}

\begin{theorem}\label{lap1}
Let $G$ be a connected graph of order $n\geq3$ and minimum degree $\delta,$ and let
$k$ be a real number belonging to $[0, n)$. If $\ q_{1}(G)<
\frac{2n\delta}{n-k}$, then $\alpha'_{*}(G)>\frac{n-k}{2}.$
\end{theorem}

\begin{proof}We distinguish the following two cases to prove.

\noindent{\bf Case 1.} $\delta>\frac{n-k}{2}$.

If $\delta\geq \frac{n}{2}$, by Lemma \ref{dirac} (1), then $G$ contains $C_{n}$. Construct a
fractional matching $f$ of $G$: $f(e)=\frac{1}{2}$ for any $e\in E(C_{n})$,
otherwise $f(e)=0$. Then we have
$$\alpha'_{*}(G)\geq \sum_{e\in E(G)}f(e)=\frac{n}{2}>
\frac{n-k}{2}.$$ Otherwise $\frac{n}{2}>\delta>\frac{n-k}{2}.$ If $G$ is 2-connected, by Lemma \ref{dirac} (2),
then $G$ contains a cycle of length at least $2\delta$.
Taking $\frac{1}{2}$ for each edge of the cycle and 0 for other edges of $G$, it is easy to see that
$$\alpha'_{*}(G)\geq\delta>\frac{n-k}{2}.$$
Otherwise $G$ contains a cut vertex $u.$ We can assume that $V_{1}$ and $V_{2}$ are two of
connected components of $G-u$. Let $v_{1}$ and $v_{2}$ be two of neighbors of $u,$ where $v_{1}\in V_{1}$ and $v_{2}\in V_{2}.$
Note that the minimum degree of the vertices in $V_{1}$ or $V_{2}$
is $\delta-1.$ Hence induced subgraphs $G[V_{1}]$ and $G[V_{2}]$ contain paths $v_{1}'P_{1}v_{1}$ and $v_{2}P_{2}v_{2}'$ of length $\delta-1,$ respectively.
So we can find a path $v_{1}'P_{1}v_{1}uv_{2}P_{2}v_{2}'$ of length $2\delta$ in $G.$ Taking $\frac{1}{2}$ for each edge of the path and 0 for others of $G$,
then we have
$$\alpha'_{*}(G)\geq\delta>\frac{n-k}{2}.$$

\noindent{\bf Case 2.} $\delta\leq\frac{n-k}{2}$.

On the contrary, suppose that $\alpha'_{*}(G)\leq\frac{n-k}{2}$. By Equation (1), there exists a vertex subset $S\subseteq V(G)$ satisfying $i(G-S)-|S|\geq \lceil k\rceil$. Let $T$ be the set of isolated vertices
in $G-S$. Note that the neighbors of each vertex of $T$ only belong to $S,$ then there exist at least $\delta$ vertices in $S.$
Let $|T|=t$ and $|S|=s,$ then
$t=i(G-S)\geq s+\lceil k\rceil.$ Consider the bipartite subgraph $H$ with the partitions $S$ and $T.$ Let $a$ be the number of edges having one end-vertex in $S$
and the other in $T.$ Then $a\geq t\delta$. For the partitions $S$ and $T$ in $H,$ the quotient matrix of $Q(H)$ is
\begin{equation*}\begin{split}
R&=\left(\begin{array}{ccccccc}
\frac{a}{s}&\frac{a}{s}\\[1mm]
\frac{a}{t}&\frac{a}{t}\\
\end{array}\right).
\end{split}\end{equation*}
By a simple calculation, the largest eigenvalue of $R$ is $\lambda_{1}(R)=\frac{a}{s}+\frac{a}{t}$. According to Lemmas \ref{q-in} and \ref{le2},
we have \begin{equation}
q_{1}(G)\geq  q_{1}(H)\geq \lambda_{1}(R)\geq \delta(\frac{t}{s}+1)\geq\delta(\frac{\lceil k\rceil}{s}+2)
\geq\delta(\frac{2\lceil k\rceil}{n-\lceil k\rceil}+2)=\frac{2n\delta}{n-\lceil k\rceil}, \label{e8}
\end{equation}
since $a\geq t\delta, t\geq s+\lceil k\rceil$ and $n\geq t+s\geq2s+\lceil k\rceil.$
This yields a contradiction. \hspace*{\fill}$\Box$
\end{proof}

From Lemmas \ref{le8} and \ref{le9}, we can see that there exists graphs $H$ with minimum degree $\delta$ such that
$\alpha'_{*}(H)=\frac{n-k}{2}$ and $q_{1}(H)=\frac{2n\delta}{n-k}.$ In the sense that Theorem \ref{lap1} is best possible.
As an application, we obtain the following relationship between $\alpha'_{*}(G)$ and $q_{1}(G)$ of graphs involving
minimum degree $\delta.$

\begin{theorem}\label{001}
Let $G$ be a connected graph of order $n\geq3$ with minimum degree $\delta$. Then
\begin{equation*}
\alpha'_{*}(G)\geq \frac{n\delta}{\ q_{1}(G)}, \label{e3}
\end{equation*}
with the equality holding if and only if $G\in \mathcal{B}(\delta, k)$ and $k=n-\frac{2n\delta}{q_{1}(G)}$ is an integer.
\end{theorem}

\begin{proof}
For simplicity, we denote $q_{1}(G)$ and $\alpha'_{*}(G)$ by $q_{1}$ and $\alpha'_{*}.$
Theorem \ref{lap1} tells us that if $q_{1}<\frac{2n\delta}{n-k}$, then $\alpha'_{*}>\frac{n-k}{2}.$
Note that $\frac{1}{n-x}$ is an increasing function of $x$ on $[0, n),$ $\frac{2n\delta}{n-k}$ decreases towards
$q_{1}$ as $k$ decreases towards $k^{*},$ where $k^{*}=n-\frac{2n\delta}{q_{1}}.$ By Theorem \ref{lap1}, $\alpha'_{*}>\frac{n-k}{2}$
for each value of $k\in(k^{*}, n).$ Letting $k$ tent to $k^{*}$ and finally equal to $k^{*},$ we obtain $\alpha'_{*}\geq\frac{n-k^{*}}{2}= \frac{n\delta}{\ q_{1}},$ as desired.

Next we consider the equality. If $G\in \mathcal{B}(\delta, k)$ and $k=n-\frac{2n\delta}{q_{1}(G)}$ is an integer. By Lemma \ref{le8}, $\alpha'_{*}(G)=\frac{n-k}{2}=\frac{n\delta}{q_{1}}.$

For the necessity, suppose that $\alpha'_{*}(G)=\frac{n\delta}{q_{1}}.$ This needs equality in each step of the above proof. Hence we have $k=k^{*}=n-\frac{2n\delta}{q_{1}},$ that is to say, $q_{1}=\frac{2n\delta}{n-k}.$
Hence the each step in (\ref{e8}) of the proof of Theorem \ref{lap1} must be an equality. Since $\lceil k\rceil=k,$ $k=n-\frac{2n\delta}{q_{1}(G)}$ must be an integer. Furthermore, since $a=t\delta, t=k+s$ and $n=2s+k$ in (\ref{e8}) of Theorem \ref{lap1}, and $s\geq \delta,$ then $G\in \mathcal{B}(\delta, k).$
\hspace*{\fill}$\Box$
\end{proof}

\section{Relations between fractional perfect matchings and $q_{1}(G)$ of graphs involving minimum degree}
~~~Using Theorem \ref{lap1}, we obtain the following sufficient condition for existence of a fractional perfect matching in a graph $G$ in terms of the signless Laplacian spectral radius of $G.$

\begin{theorem}\label{th10}
Let $G$ be a connected graph of order $n\geq3$ with minimum degree $\delta$. If
$q_{1}(G)< \frac{2n\delta}{n-1}$, then $G$ has a fractional
perfect matching.
\end{theorem}

\begin{proof}
By Theorem \ref{lap1}, taking $k=1$, then we have $\alpha'_{*}(G)>\frac{n-1}{2}.$
Recalling that $2\alpha'_{*}(G)$ is an integer, then $\alpha'_{*}(G)=\frac{n}{2}.$ Thus $G$ has a fractional
perfect matching.\hspace*{\fill}$\Box$
\end{proof}

Further, we consider the sufficient condition for existence of a fractional perfect matching in a graph $G$ in terms of the signless Laplacian spectral radius of its complement.

\begin{theorem}\label{girth}
Let $G$ be a connected graph of order $n\geq3$ with minimum degree $\delta,$ and $G^{c}$ be the complement of $G$. If
$q_{1}(G^{c})<2\delta$, then $G$ has a fractional perfect
matching.
\end{theorem}

\begin{proof}
Suppose, for the sake of contradiction, that $\alpha'_{*}(G)<n/2$. By Equation (\ref{e1}), there
exist a vertex set $S\subseteq V(G)$ such that $i(G-S)-|S|>0$. Let
$T$ be the set of isolated vertices in $G-S$. Since the
neighbors of each vertex of $T$ only belong to $S$, there are at
least $\delta$ vertices in $S$. Hence $|T|\geq |S|+1\geq \delta+1$.
Note that $G^{c}[T]$ is a clique. By Lemma \ref{le2}, we
have
$$q_{1}(G^{c})\geq q_{1}(G^{c}[T])=2(|T|-1)\geq 2\delta,$$
a contradiction. \hspace*{\fill}$\Box$
\end{proof}

Consider the complete bipartite graph $K_{\delta+1, \delta}.$ Clearly
$q_{1}(K_{\delta+1, \delta}^{c})=2\delta$. However, $K_{\delta+1, \delta}$ has no fractional perfect
matching since $\alpha'_{*}(K_{\delta+1, \delta})=\alpha'(K_{\delta+1, \delta})=\delta<\frac{2\delta+1}{2}.$
In this sense that the result of Theorem \ref{girth} is best possible.

In fact, the sufficient condition in Theorem \ref{girth} is very useful and important. There exist some graphs which can be checked for existence of a fractional perfect matching
only by applying Theorem \ref{girth}, not Theorem \ref{th10}.

\begin{example}
Let $H$ be a graph obtained by joining $t$ edges between an isolated vertex and the complete graph $K_{2t},$ where $2\leq t<2t.$ Note that $H$ contains $C_{2t+1},$ hence $H$ has a fractional perfect matching. However by Lemma \ref{le3},
$$q_{1}(H)\geq 2\overline{d}(H)\geq 2\cdot\frac{t+t\cdot(2t-1)+t\cdot2t}{2t+1}=\frac{8t^{2}}{2t+1}>2t+1=\frac{2|V(H)|\delta(H)}{|V(H)|-1}.$$
Clearly, we can not use the sufficient condition in Theorem \ref{th10} to verify. Note that $$q_{1}(H^{c})=q_{1}(K_{1, t})=t+1<2t=2\delta(H),$$
so the sufficient condition in Theorem \ref{girth} is available.
\end{example}

Indeed, if we exclude some graphs, the upper bound in Theorem \ref{girth} can be slightly raised.
Let $H=(\delta+1)K_{1}\bigtriangledown H_{\delta},$ where $H_{\delta}$ be any graph of order $\delta$.
Clearly $q_{1}(H^{c})=2\delta.$ Note that there exists a vertex subset $V(H_{\delta})$ such that
$i(H-V(H_{\delta}))>|V(H_{\delta})|$, hence $H$ has no fractional perfect
matching.

\begin{theorem}
Let $G$ be a connected graph on $n\geq3$ vertices with minimum degree
$\delta$, and $G^{c}$ be the complement of $G$. If
$q_{1}(G^{c})<2\delta+2$, then $G$ has a fractional perfect
matching unless $G\cong (\delta+1)K_{1}\bigtriangledown H_{\delta}.$
\end{theorem}

\begin{proof}
Suppose that $\alpha'_{*}(G)<n/2$. By Equation (\ref{e1}), there exist a
vertex set $S\subseteq V(G)$ such that $i(G-S)-|S|>0$. Let $T$ be
the set of isolated vertices in $G-S$. Note that the
neighbors of each vertex of $T$ only belong to $S,$ then $|T|\geq
|S|+1\geq \delta+1$.

We claim that $V(G)=T\cup S.$ Otherwise we can find a
clique of order $\delta+2$ in $G^{c}$, and $q_{1}(G^{c})\geq
2\delta+2$, a contradiction. If $|T|\geq \delta+2$, then there is
a clique of order $\delta+2$ in $G^{c}$, and $q_{1}(G^{c})\geq
2\delta+2$, a contradiction. Hence $|T|=\delta+1$, and thus $|S|=\delta$,
that is to say, $G\cong (\delta+1)K_{1}\bigtriangledown H_{\delta}.$ \hspace*{\fill}$\Box$
\end{proof}

\small {

}
\end{document}